\documentclass[12pt, a4paper, abstracton]{scrartcl}
\pdfoutput=1

\usepackage{amsmath, amsthm, amsfonts, amssymb}
\usepackage[utf8]{inputenc}
\usepackage[nottoc]{tocbibind} 
\usepackage{color}
\usepackage{slashed} 
\usepackage[all, cmtip]{xy}
\usepackage{hyperref} 
\usepackage{chngcntr} 
\usepackage{graphicx} 
\usepackage[english]{babel}
\usepackage{microtype}
\usepackage{bm} 
\usepackage{xfrac} 
\usepackage{mathtools} 
\usepackage{mathabx}

\newcommand{\IN}{\mathbb{N}}

\newcommand{\IR}{\mathbb{R}}
\newcommand{\IC}{\mathbb{C}}

\newcommand{\calS}{\mathcal{S}}

\newcommand{\IB}{\mathfrak{B}}
\newcommand{\IK}{\mathfrak{K}}

\newcommand{\supp}{\operatorname{supp}}
\newcommand{\diam}{\operatorname{diam}}

\newcommand{\id}{\operatorname{id}}

\newcommand{\card}{\#}

\newcommand{\trace}{\operatorname{tr}}
\newcommand{\ch}{\operatorname{ch}}

\newcommand{\op}{\mathrm{op}}

\newcommand{\Cpoln}{C^\ast_{\mathrm{pol},n}}
\newcommand{\Cpol}{C^\ast_{\mathrm{pol}}}
\newcommand{\Bpol}{B_{\mathrm{pol}}}

\newcommand{\HCocont}{H \! C^\mathrm{cont}}

\newcommand{\HC}{H \! C}
\newcommand{\HP}{H \! P}

\newcommand{\Huf}{H^\mathrm{uf}}

\newcommand{\Hufpol}{H^\mathrm{pol}}

\newcommand{\HufG}{H^{\mathrm{uf},G}}

\newcommand{\HpolG}{H^{\mathrm{pol},G}}
\newcommand{\CpolG}{C^{\mathrm{pol},G}}
\newcommand{\CufG}{C^{\mathrm{uf},G}}

\newcommand{\uf}{\mathrm{uf}}

\newcommand{\Frechet}{Fr\'{e}chet }

\newcommand{\alg}{\mathrm{alg}}

\newcommand{\CAT}{\mathrm{CAT}}

\newcommand{\KH}{K \! H}

\newcommand{\Vol}{\operatorname{Vol}}
\newcommand{\FVol}{\operatorname{FVol}}
\newcommand{\Diam}{\operatorname{Diam}}
\newcommand{\FDiam}{\operatorname{FDiam}}

\newcommand{\asCone}{\operatorname{asCone}}
\newcommand{\omegalim}{\operatorname*{\omega-lim}}
\newcommand{\dist}{\operatorname{dist}}
\newcommand{\idrm}{\mathrm{id}}

\newcommand*{\largecdot}{\raisebox{-0.25ex}{\scalebox{1.2}{$\cdot$}}}

\newtheorem{thm}{Theorem}[section]
\newtheorem*{thm*}{Theorem}
\newtheorem*{mainthm*}{Main Theorem}
\newtheorem{cor}[thm]{Corollary}
\newtheorem*{cor*}{Corollary}
\newtheorem{lem}[thm]{Lemma}
\newtheorem{prop}[thm]{Proposition}

\newtheorem*{conj*}{Conjecture}
\newtheorem*{obs*}{Observation}

\theoremstyle{definition}
\newtheorem{rem}[thm]{Remark}
\newtheorem{attention}[thm]{Attention}
\newtheorem{example}[thm]{Example}

\newtheorem{defn}[thm]{Definition}

\numberwithin{equation}{section}

\counterwithout{footnote}{section}

\makeatletter
\def\blfootnote{\gdef\@thefnmark{}\@footnotetext}
\makeatother

\begin{document}

\title{Banach strong Novikov conjecture for polynomially contractible groups}
\author{Alexander Engel
}
\date{}
\maketitle

\vspace*{-3.0\baselineskip}
\begin{center}
\footnotesize{
\textit{
Fakult\"{a}t f\"{u}r Mathematik\\
Universit\"{a}t Regensburg\\
93040 Regensburg, GERMANY\\
\href{mailto:alexander.engel@mathematik.uni-regensburg.de}{alexander.engel@mathematik.uni-regensburg.de}
}}
\end{center}

\begin{abstract}
We prove the Banach strong Novikov conjecture for groups having polynomially bounded higher-order combinatorial functions. This includes all automatic groups.
\end{abstract}

\tableofcontents

\paragraph{Acknowledgements} The author was supported by the SFB 1085 ``Higher Invariants'' and the Research Fellowship EN 1163/1-1 ``Mapping Analysis to Homology'', both funded by the Deutsche Forschungsgemeinschaft~DFG.

The author is grateful to Cornelia Dru\c{t}u, Bogdan Nica, Tim Riley and Guoliang Yu for answering a lot of questions, and to Clara L\"{o}h for many helpful comments. He thanks S\"{u}leyman Ka\v{g}an Samurka\c{s} for pointing out a mistake in a first version of this article, and the anonymous referee for his or her remarks.

\section{Introduction}

Let $G$ be a discrete, countable group. In the first of the two main theorems we denote
\[d := \sup\{n\colon H_n(G;\IC) \not= 0\}\]
and assume $d < \infty$.
The other occuring notions will be explained afterwards.

\paragraph{Main Theorem I (finite-dimensional case)}
\emph{Let the group $G$ be of type $F_{d+1}$ and let it be polynomially contractible up to order $d$.}

\emph{Then the Banach strong Novikov conjecture holds for $G$ with exponent $p = \sfrac{d+2}{d+1}$, i.e., the analytic assembly map $RK_\ast(BG) \to K_\ast(B_r^p G)$ is rationally injective, where $B_r^p G$ denotes the closure of $\IC G \subset \IB(\ell^p G)$.}

\paragraph{Main Theorem II}

\emph{Let the group $G$ be of type $F_\infty$ and let it be polynomially contractible.}

\begin{enumerate}
\item \emph{Then the Bost assembly map $RK_\ast(BG) \to K_\ast(\ell^1 G)$ is rationally injective.}

\item \emph{If $G$ also has Property (RD), then the strong Novikov conjecture holds for $G$, i.e., the analytic assembly map $RK_\ast(BG) \to K_\ast(C_r^\ast G)$ is rationally injective.}
\end{enumerate}

\paragraph{Definitions}
\begin{itemize}
\item Recall that for a countable and discrete group $G$ its classifying space $BG$ is, up to homotopy equivalence, uniquely determined by requiring that $\pi_1(BG) \cong G$ and that the universal cover of $BG$ is contractible.

\item The group $G$ is called of type $F_N$ if it admits a CW-complex model for $BG$ consisting of finitely many cells in each dimension up to dimension $N$. This is equivalent to admitting a model which is a simplicial complex and consists of finitely many simplices of each dimension up to dimension $N$. A group is of type $F_\infty$, if it admits a model for $BG$ consisting of finitely many cells (resp., simplices) in each dimension.

\item Due to several competing definitions of higher-order combinatorial functions and since our proofs work with either choice, we will give two variants for the notion of polynomial contractibility (concrete definitions will be given in Section~\ref{secn2339992}):
\begin{enumerate}
\item We call a group of type $F_{N+1}$ polynomially contractible up to order $N$, if for every $n \le N$ its higher-order two-variable isoperimetric and isodiametric functions $\delta^{(n)}(-,-)$ and $\eta^{(n)}(-,-)$ in the sense of Riley \cite{riley} are bounded from above by polynomials in two variables.

A group of type $F_\infty$ is called polynomially contractible, if it is polynomially contractible up to order $N$ for all $N \in \IN$.

\item The variant of the definition of polynomial contractibility is to use higher-order Dehn functions $d^n(-)$ as defined by Ji--Ramsey \cite{ji_ramsey}. Note that they depend only on one variable.
\end{enumerate}

\item Let $G$ be a finitely generated group. Fixing a generating set for $G$, we denote by $|\! -\! |$ the word-length derived from it. We denote by $H^s G$ the complex Banach space of all functions $\phi$ on $G$ with $(1+|\! -\! |)^s \phi(-) \in \ell^2 G$, and we put on $H^\infty G := \bigcap_{s \in \IR} H^s G$ the obvious \Frechet topology. We say $G$ has Property~(RD), if we have a continuous inclusion $H^\infty G \hookrightarrow C_r^\ast G$, where the latter is the reduced group $C^\ast$-algebra.
\end{itemize}

\paragraph{The class of polynomially contractible groups} Let us summarize which other classes of groups the class of polynomially contractible groups encompasses.

We assume in the following diagram that all groups are finitely generated, and then it follows automatically for these groups that they will be of type $F_\infty$: for groups with contractible asymptotic cones this was proven by Riley \cite[Theorem D]{riley} and for the combable groups by Alonso \cite[Theorem 2]{alonso}.

In Section~\ref{secn2339992} we will give references, resp., (ideas of) proofs for many of the implication arrows in the following diagram. We will also discuss in Section~\ref{secn2339992} the definitions of most of the occuring notions.

The equivalence of polynomial contractibility to polynomially bounded cohomology will be discussed in the next paragraph. Note that by the notion ``polynomially bounded cohomology'' we mean here that we have polynomial bounds for all coefficient modules, and not only coefficients $\IC$ as in Connes--Moscovici \cite{connes_moscovici}.

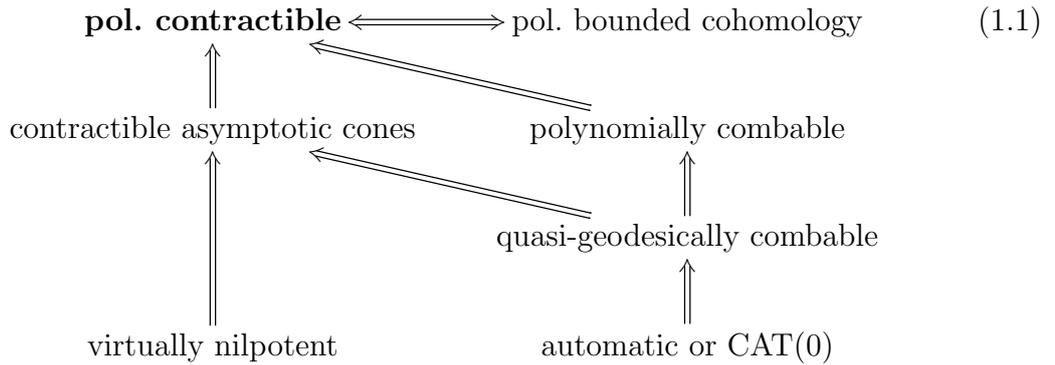
\begin{equation}
\label{eq875562}
\xymatrix{
\textbf{pol.~contractible} & \text{pol.~bounded~cohomology} \ar@{<=>}[l]\\
\text{contractible~asymptotic~cones} \ar@{=>}[u] & \text{polynomially~combable} \ar@{=>}[ul]\\
& \text{quasi-geodesically~combable} \ar@{=>}[u] \ar@{=>}[ul]\\
\text{virtually nilpotent} \ar@{=>}[uu] & \text{automatic or } \CAT(0) \ar@{=>}[u]}
\end{equation}

To the knowledge of the author, up to now the Banach strong Novikov conjecture was not known for the class of automatic groups. This is a rich class of groups and contains, e.g., the following other classes of groups (we give in parentheses references for the proofs that the corresponding groups are indeed automatic).
\begin{itemize}
\item hyperbolic groups (see, e.g., \cite[Theorem 3.4.5]{epstein_et_al})
\item Coxeter groups (Brink--Howlett \cite{brink_howlett})
\item Artin groups of finite type (Charney \cite{charney})
\item systolic groups (Januszkiewicz--\'{S}wi{\c{a}}tkowski \cite[Theorem E]{janus_swia})
\item mapping class groups (Mosher \cite{mosher})
\item $\CAT(0)$ cube groups (Niblo--Reeves \cite[Theorem 5.3]{niblo_reeves} for torsion-free groups and {\'{S}}wi{\c{a}}tkowski \cite[Corollary 8.1]{swiatkowski} in the general case)
\end{itemize}

\paragraph{Relation to polynomially bounded cohomology} Both Meyer \cite{meyer} and Ogle \cite{ogle_pol} proved that if a group is polynomially combable, then it will have polynomially bounded cohomology. More generally, Ogle (in loc.~cit.) proves that a group has polynomially bounded cohomology if its higher-order Dehn functions are polynomially bounded.

Ogle uses a different definition for the higher-order Dehn functions than we do in this paper. Ji--Ramsey \cite{ji_ramsey} introduce the definition for higher-order Dehn functions that we use, and then they show that if all higher-order Dehn functions (in their definition) are polynomially bounded, then this also holds for Ogle's version of higher-order Dehn functions, and vice versa. And this in turn, as Ji--Ramsey show, is equivalent to Gersten's version of higher-order Dehn functions being polynomially bounded \cite{gersten}.

The main result of Ji--Ramsey is then that a group of type $F_\infty$ has polynomially bounded higher-order Dehn functions if and only if it has polynomially bounded cohomology for all coefficient modules.

So the equivalence in Diagram \eqref{eq875562} holds only for the version of polynomial contractibility that uses the higher-order Dehn functions of Ji--Ramsey. But morally different notions of higher-order combinatorial functions should be comparable. Therefore it is a natural question whether having polynomially bounded cohomology for all coefficients is equivalent to polynomial contractibility using Riley's higher-order functions.

\subsection{Related work}
\label{subsec8988}

Being one of the most important conjectures in geometry and topology of manifolds, the strong Novikov conjecture received of course a lot of attention in the last few decades. The Banach version of the strong Novikov conjecture did not get much attention at all, mainly because it has no known implication to the geometry and topology of manifolds. Let us relate in this section the results of the present paper to some earlier results.

\paragraph{Relation to the result of Connes--Moscovici} The strong Novikov conjecture was proven by Connes--Moscovici \cite{connes_moscovici} for groups having Property (RD) and polynomially bounded cohomology (for coefficients $\IC$). The main example of groups having these two properties are hyperbolic groups.

From the discussion in the previous section it follows that, very roughly\footnote{That is to say, ignoring the difference between having polynomially bounded cohomology for all coefficients or only for coefficients $\IC$.}, one might say that the main work of this paper is to lift from the result of Connes--Moscovici the assumption that the group must have Property (RD). But to still be able to deduce a strong Novikov type statement, we must in exchange pass from the case $p=2$ to the case $1<p<2$ on the right hand side of the assembly map. In the presence of Property~(RD) we recover Connes--Moscovici's original result.

The question arises for which groups the $K$-theories of $B^p_r G$ and of $C_r^\ast G$ are isomorphic. To the knowledge of the author there are currently no published results in the literature. For hyperbolic groups $G$ work-in-progress of Liao--Yu \cite{liao_yu} shows that the $K$-theory of $B^p_r G$ is independent of $1 \le p < \infty$. We will discuss this in greater detail in Remark~\ref{remui3489re}.

Referring to Chatterji's overview article \cite{chatterji} about Property (RD), the following groups from Diagram \eqref{eq875562} are known to enjoy it: virtually nilpotent, hyperbolic, Coxeter, $\CAT(0)$ cubical, and mapping class groups (Sapir \cite{sapir_RD} also wrote a nice overview article about it). To the knowledge of the author there seems to be no relation between having contractible asymptotic cones and having Property (RD).

\paragraph{Relation to Yu's work} Yu proved the coarse Baum--Connes conjecture for groups which coarsely embed into a Hilbert space \cite{yu_embedding_Hilbert_space}. This implies by the descent principle the strong Novikov conjecture for such groups, provided they admit a finite classifying space. Later Skandalis--Tu--Yu proved injectivity of the Baum--Connes assembly map with coefficients for all coarsely embeddable groups \cite{skandalis_tu_yu}.

From the classes of groups we discussed in \eqref{eq875562}, virtually nilpotent groups and many automatic groups are known to be coarsely embeddable. But for the class of automatic groups this is currently not known. Although for automatic groups one might conjecture that they should be coarsely embedabble, it seems unlikely that a more general class of groups, e.g., groups with contractible asymptotic cones, should enjoy this property.

\subsection{Related open questions and problems}

Let us collect some open questions and problems that the author thinks are important:
\begin{enumerate}
\item Are both versions of polynomial contractibility equivalent, i.e., can one prove that a group of type $F_\infty$ has polynomially bounded higher-order two-variable isoperimetric and isodiametric functions in the sense of Riley if and only if it has polynomially bounded higher-order Dehn functions in the sense of Ji--Ramsey?
\item Do there exist automatic groups without Property (RD)?
\item Do there exist Property (RD) groups which are not polynomially contractible?
\item Are automatic groups coarsely embeddable into a Hilbert space?
\item Can one construct quasi-geodesically combable groups which do not admit coarse embeddings into a Hilbert space?
\item For which groups is $K_\ast(B^p_r G)$ independent of $p\in (1,\infty)$?

A positive answer in the case of polynomially contractible groups would prove the strong Novikov conjecture for these groups by the results in this paper.
\end{enumerate}

\subsection{Idea of the proof of the main theorem}
\label{seck23345}

\paragraph{Outline of the general argument}

Since the Banach strong Novikov conjecture only asks about rational injectivity, the idea is to pass via Chern characters to homology theories (since the ordinary topological Chern character is rationally an isomorphism).

\begin{equation}
\label{eqnjk23d}
\xymatrix{
RK_\ast(BG) \ar[rr]^{\mu^\alg} \ar@/^2pc/[rrr]^{\text{analytic assembly map}} \ar[dd]_{\ch_n} & & K^\alg_\ast(\calS G) \ar[r] \ar[d]_{\trace \circ \ch_n^{\calS G}} & K_\ast(B_r^p G) \ar[d]\\
& & \HC_n(\IC G) \ar[r] \ar[d]_{\chi_n} & \HCocont_n(\Bpol^p G) \ar@{-->}[d]\\
H_{n}(BG) \ar[rr] & & \HufG_{n}(G) \ar[r] & \HpolG_{n}(G)}
\end{equation}

We will discuss the above diagram in more detail in Section~\ref{seco23324545}. The Chern character $\ch_n\colon RK_\ast(BG) \to H_{n}(BG)$ is rationally injective. So rational injectivity of the analytic assembly map (on classes that map to degree $n$ classes in homology) will follow from injectivity of the lower row. One of the main steps in the proof of the main theorem is therefore to show that for polynomially contractible groups this map is injective. We will even prove that in this case it is an isomorphism, see Corollary~\ref{corj9809uk23ds}.

Another step in the proof of the main theorem is to actually construct the diagram. Let us highlight here first the dense subalgebra $\Bpol^p G \subset B^p_r G$, which will be introduced in Section~\ref{sec90234}. Its main technical property is that it is closed under holomorphic functional calculus (see Propositions~\ref{prop239ds} and \ref{prop89034ew}) and therefore has the same $K$-theory as $B_r^p G$. The main analytical part of our proof is to derive good kernel estimates for operators from this new subalgebra, see Section~\ref{seckn234323} and Lemma~\ref{lem8943e}. For example,  we will prove that $\Bpol^p G$ is continuously included in the $\ell^p$-space of rapidly decreasing functions.

It is a crucial ingredient to show that the map $\chi_n$ is continuous and so extends to the dotted map in the diagram. The homology groups $\HpolG_{n}(G)$ are defined by completing the chain complex $\CufG_\bullet(G)$ under a family of norms (Definition~\ref{defnnorm}). Now on the one hand we want these norms to be as large as possible so that it is easier to show that the map $H_{n}(BG) \to \HpolG_{n}(G)$ is bounded from below (from which we will deduce that it is actually an isomorphism). But on the other hand, we want these norms on $\CufG_\bullet(G)$ to be as small as possible so that it is easier for us to show that the map $\chi_n$ is continuous.

Similar diagrams like \eqref{eqnjk23d}, resp., the corresponding underlying idea were already considered before. One can divide the left square into two by introducing the corresponding assembly map for periodic cyclic homology \cite[Remark~2.9]{EM_burghelea}. The resulting upper left square was investigated by Corti\~{n}as--Tartaglia \cite{ctartaglia}. Similar diagrams were also considered by Ji--Ramsey \cite[Page 38]{JR} and one can argue that the basic idea goes back to Connes--Moscovici \cite[Theorem 5.4]{connes_moscovici}. The idea to map assembly maps to homology was also already employed by Yu \cite{yu_algebraic_novikov}. The computation of the homology of group rings was carried out by Burghelea \cite{burghelea}, and the computations of homology groups in relation to isomorphism conjectures by L\"{u}ck--Reich \cite[Theorem 0.7]{lueck_reich}.

\section{Quasi-local completions of group rings}
\label{sec90234}

This section contains the analytical meat of our argument. We will first introduce and discuss the subalgebra $\Cpol G \subset C_r^\ast G$, and then derive in Section~\ref{seckn234323} important kernel estimates for operators from $\Cpol G$.

Let $G$ be any countable, discrete group and denote by $\IC G$ the complex group ring. Recall that the reduced group $C^\ast$-algebra $C_r^\ast G$ is the closure of $\IC G \subset \IB(\ell^2 G)$.

Equip $G$ with any proper, left-invariant metric. We will denote the resulting metric space also by $G$. Note that any other proper, left-invariant metric on $G$ results in a space which is coarsely equivalent to the previous one. In fact, the identity map will be in this case a coarse equivalence.

Recall that an operator $A \in \IB(\ell^2 G)$ has finite propagation if there exists an $R > 0$, such that $\supp Au \subset B_R(\supp u)$ for all $u \in \ell^2 G$. Here we denote by $B_R(-)$ the ball of radius $R$. It is immediate that having finite propagation does not depend on the concrete choice of proper, left-invariant metric on $G$.

\begin{defn}[Quasi-local operators, cf.~Roe {\cite[Section 5]{roe_index_1}}]\label{defnoi34efr}
We call an operator $A \in \IB(\ell^2 G)$ quasi-local if there is a function $\mu\colon \IR_{> 0} \to \IR_{\ge 0}$ with $\mu(R) \to 0$ for $R \to \infty$ and such that
\[\|Au\|_{G \setminus B_R(\supp u)} \le \mu(R) \cdot \|u\|\]
for all $u \in \ell^2 G$.

We call such a function $\mu$ a dominating function for $A$.
\qed
\end{defn}

Note that any finite propagation operator is quasi-local, and analogously to the finite propagation case being quasi-local does not depend on the choice of proper, left-invariant metric on $G$.

By the following lemma we conclude that every operator from $C_r^\ast G$ is quasi-local:

\begin{lem}[{\cite[Lemma~2.26]{engel_phd}}]\label{lemkjdsf0923}
Let $(A_i)_{i \in \IN}$ be a sequence of quasi-local operators converging in operator norm to an operator $A$. Then $A$ is also quasi-local.
\end{lem}

\begin{proof}
By repeating some of the operators in the sequence $A_i$ we can assume that the propagation of $A_i$ is at most $i$. Then the dominating function of $A$ can be bounded from above by $\mu_A(i) \le \|A-A_i\|_\op$ for $i \in \IN$, which goes to $0$ as $i \to \infty$ by assumption.
\end{proof}

\begin{attention}
It is tempting to think that quasi-local operators are approximable by finite propagation operators. But it is not known whether this is indeed always the case, cf.~the discussion in \cite[Section~6]{engel_indices_UPDO}.

To the knowledge of the author, there is only the result of Rabinovich--Roch--Silbermann \cite{RRS}, resp., of Lange--Rabinovich \cite{lange_rabinovich} that on $\IR^n$ every quasi-local operator is approximable by finite propagation operators.

The other (partial) result that the author knows is his own \cite[Corollary 2.33]{engel_rough} that on spaces of polynomial growth one can approximate operators with a super-polynomially fast decaying dominating function by finite propagation operators.
\qed
\end{attention}

\begin{defn}\label{defn92434}
Let $G$ be a finitely generated group and fix a word-metric on $G$ with respect to a choice of a finite generating set.

For every $n \in \IN$ we define a norm on $\IC G \subset \IB(\ell^2 G)$ by
\begin{equation}
\label{eqnjk23vd}
\|A\|_{\mu,n} := \inf\{D>0\colon \mu_A(R) \le D/R^n \ \forall R > 1\},
\end{equation}
where $\mu_A(R)$ denotes the smallest possible dominating function for $A \in \IC G$, i.e., for every $R > 1$ we have
\begin{equation}
\label{eq12345asdfg}
\mu_A(R) = \inf\{C>0\colon \|Au\|_{G \setminus B_R(\supp u)} \le C \cdot \|u\| \text{ for all } u \in \ell^2(G)\}.
\end{equation}

We let $\Cpol G$ be the closure of $\IC G$ under the family of norms $(\|-\|_\op, \|-\|_{\mu,n}, \|-^\ast\|_{\mu,n})$ for all $n \in \IN$.
\qed
\end{defn}

\begin{rem}\label{remji2323}
The reason why we restrict us in the above definition to finitely generated groups is because any other choice of a finite generating set will result in a word-metric which is quasi-isometric to the previously chosen one. This results in a direct comparison between the norms $\|-\|_{\mu,n}$ for different choices of metrics, and therefore the algebra $\Cpol G$ becomes an invariant of the group $G$ itself.

If $G$ is not finitely generated, different choices of proper, left-invariant metrics lead in general only to coarsely equivalent spaces. If such a coarse equivalence now distorts lengths of elements of $G$ super-polynomially, then it is not clear anymore if the algebra $\Cpol G$ is independent of the choice of metric.
\qed
\end{rem}

In order to understand why $\Cpol G$ is an algebra, i.e., why it is closed under composition, we need the following estimate of Roe \cite[Proposition~5.2]{roe_index_1}: if $\mu_A$ denotes a dominating function for the operator $A \in \IB(\ell^2 G)$ and $\mu_B$ one for $B \in \IB(\ell^2 G)$, then a dominating function for the composition $AB$ is given by
\begin{equation}
\label{eq2334566}
\mu_{AB}(R) \le \|A\|_\op \cdot 2\mu_B(R/2) + \mu_A(R/2) \big( \|B\|_\op + 2\mu_B(R/2) \big).
\end{equation}
Note that in the case of rapidly decreasing functions it is only known that they constitute a convolution algebra if the group has Property (RD), see Jolissaint \cite[Lemma~1.2.4]{jolissaint_2}.

By induction over \eqref{eq2334566} we can show the estimate
\begin{equation}
\label{eq8934rg5}
\mu_{A^{n+1}}(R) \le \sum_{k=1}^n 5^k \|A\|^n_\op \mu_A(R/2^k)
\end{equation}
for all $R > 1$ and every $n \in \IN$.

The following is the main technical result, namely that $\Cpol G$ is closed under holomorphic functional calculus as well as all matrix algebras over it (we call this property ``smoothness''). The arguments are analogous to the ones in \cite[Section~2.3]{engel_rough}, where the corresponding statement was shown for a certain version of $\Cpol G$ on non-compact Riemannian manifolds.

\begin{prop}\label{prop239ds}
$\Cpol G$ is a dense and smooth \Frechet $^\ast$-subalgebra\footnote{A \Frechet $^\ast$-algebra is an algebra with a topology turning it into a \Frechet space with jointly continuous multiplication and such that the $^\ast$-operation is continuous. Note that we do not require here that the semi-norms used to define the \Frechet topology are sub-multiplicative.} of $C_r^\ast G$.
\end{prop}

\begin{proof}
The only non-trivial point in showing that $\Cpol G$ is a dense \Frechet $^\ast$-subalgebra of $C_r^\ast G$ is to show that multiplication is jointly continuous. But this follows from \eqref{eq2334566}.

We have to show that $\Cpol G$ is closed under holomorphic functional calculus. From Schweitzer \cite[Corollary~2.3]{schweitzer} it then follows that all matrix algebras over $\Cpol G$ are also closed under holomorphic functional calculus.

By \cite[Lemma~1.2]{schweitzer} it suffices to show that $\Cpol G$ is inverse closed, and for this it suffices by \cite[Lemma~3.38]{GBVF} to show the following: there exists an $\varepsilon > 0$ such that $A \in \Cpol G$ with $A \in B_\varepsilon(\idrm)$ implies that $A$ is invertible in $\Cpol G$. Note that $B_\varepsilon(\idrm) \subset C_r^\ast G$ denotes a ball whose radius $\varepsilon$ is measured in operator norm, and the goal is to show that $\|A^{-1}\|_{\mu,n} < \infty$ for every $n \in \IN$.\footnote{Note that we actually also have to show that $A^{-1}$ lies in the closure of $\IC G$ under the norms $\|-\|_{\mu,n}$. But this will follow from the last estimate that we give in this proof, since it will show that $A^{-1}$ is approximated by $\sum_{k=0}^N (\idrm-A)^k$ for $N \to \infty$ in the norms $\|-\|_{\mu,n}$, and these operators all lie in $\Cpol G$ for all $N \in \IN$.}

Note that we will not be able to show exactly the above. We will be able to show that $\|A^{-1}\|_{\mu,n} < \infty$ for every $n \in \IN$, but our choice of $\varepsilon$ will depend on $n$. But this is also ok for us, because it shows that $\Cpoln G$ is closed under holomorphic functional calculus, where $\Cpoln G$ is the closure of $\IC G$ under the norms $(\|-\|_\op, \|-\|_{\mu,n}, \|-^\ast\|_{\mu,n})$ for a fixed $n \in \IN$. This are also \Frechet $^\ast$-algebras, and our arguments will show that they are closed under holomorphic functional calculus. It follows that the algebra $\Cpol G = \bigcap_{n \in \IN} \Cpoln G$ is also closed under holomorphic functional calculus.

Let $A \in \Cpol G$ with $A \in B_\varepsilon(\idrm)$ be given (we will fix our choice of $\varepsilon > 0$ later). We write $A = \idrm - (\idrm - A)$ and note the estimate $\|\idrm-A\|_\op < \varepsilon$. Hence we can write the inverse as $A^{-1} = (\idrm - (\idrm-A))^{-1} = \sum_{n=0}^\infty (\idrm-A)^n$ provided $\varepsilon < 1$. We start our estimate with
\[\mu_{\sum_{n=0}^N (\idrm-A)^n}(R) \le \sum_{n=0}^N \mu_{(\idrm-A)^n}(R) \le \sum_{n=0}^N \sum_{k=1}^{n-1} 5^k \|\idrm-A\|^{n-1}_\op \mu_{\idrm-A}(R/2^k),\]
where we have used \eqref{eq8934rg5}. Note that on the left hand side we can let the sum start at $n=1$ since $(\idrm-A)^0=\idrm$ has no propagation. Furthermore, we use $\|\idrm-A\|_\op < \varepsilon$ and that, again since $\id$ has no propagation, $\mu_{\idrm-A}(-) = \mu_A(-)$. So we can go on and have
\begin{align*}
\mu_{\sum_{n=0}^N (\idrm-A)^n}(R) & \le \sum_{n=1}^N \sum_{k=1}^{n-1} 5^k \varepsilon^{n-1} \mu_A(R/2^k)\\
& = \sum_{k=1}^{N-1} \sum_{n=k+1}^N 5^k \varepsilon^{n-1} \mu_A(R/2^k)\\
& \le \sum_{k=1}^{N-1} 5^k \mu_A(R/2^k) \frac{\varepsilon^{k+1}}{1-\varepsilon}\\
& \le \sum_{k=1}^{N-1} \frac{(5 \varepsilon)^{k+1}}{1-\varepsilon} \mu_A(R/2^k),
\end{align*}
where in the second-to-last inequality we assume $\varepsilon \le 1/2$. Now we use that we have $\mu_A(R/2^k) \le \|A\|_{\mu,l} 2^{kl} / R^l$ for every $l \in \IN$. We fix now an $l \in \IN$ and we fix now $\varepsilon < \frac{1}{2}\cdot\frac{1}{5\cdot 2^l}$. Then we have
\begin{align*}
\mu_{\sum_{n=0}^N (\idrm-A)^n}(R) & \le \sum_{k=1}^{N-1} \frac{(5 \varepsilon)^{k+1}}{1-\varepsilon} \cdot \frac{\|A\|_{\mu,l} 2^{kl}}{R^l}\\
& \le \frac{\|A\|_{\mu,l}}{R^l (1-\varepsilon)} \sum_{k=1}^{N-1} \underbrace{(5 \varepsilon)^{k+1} \cdot 2^{kl}}_{\le (5 \varepsilon 2^l)^{k+1}}\\
& \le \frac{2(1-(\sfrac{1}{2})^{N+1})}{1-\varepsilon} \cdot \|A\|_{\mu,l} \cdot 1/R^l.
\end{align*}
Letting $N \to \infty$ we therefore conclude that $\|A^{-1}\|_{\mu,l} < \infty$, finishing this proof.
\end{proof}

\begin{rem}\label{remui24343t5}
Our above estimate shows that $\sum_{n=0}^N (\idrm-A)^n$ approximates $A^{-1}$ exponentially fast. Furthermore, if we assume that $A$ has finite propagation, we can conclude $\mu_{A^{-1}}(R) \le \sqrt[\operatorname{prop} A]{\varepsilon}\cdot \varepsilon^R$ for $\varepsilon < 1/2$ and $A \in B_\varepsilon(\idrm)$, showing that $A^{-1}$ has an exponentially decaying dominating function. To get this estimate, we use the idea from the proof of Lemma~\ref{lemkjdsf0923} combined with the fact that $\sum_{n=0}^N (\idrm - A)^n$ approximates $A^{-1}$ as $N \to \infty$ and that we have the estimate $\operatorname{prop} \Big( \sum_{n=0}^N (\idrm - A)^n \Big) \le N \cdot \operatorname{prop} (A)$.
\qed
\end{rem}

\subsection{Kernel estimates}
\label{seckn234323}

The following lemma introduces the basic idea how to derive kernel estimates from the norms \eqref{eqnjk23vd}.

\begin{lem}[cf.~{\cite[Proposition~5.4]{roe_index_1}}]
\label{lemjk231fsd}
Let $A = \sum_{g \in G} a_g g \in \Cpol G$.

Then for every $n \in \IN$ we have
\[\sum_{g \in G \setminus B_R(e)} |a_g|^2 < \|A\|_{\mu,n}^2 / R^{2n},\]
where $B_R(e)$ denotes the ball of radius $R > 1$ around $e \in G$.
\end{lem}

\begin{proof}
Denote by $\delta_e \in \ell^2 G$ the function with value $1$ at $e \in G$. Since $A$ is a quasi-local operator we have $\|A \delta_e\|^2_{G \setminus B_R(e)} \le \mu_A(R)^2 \cdot \|\delta_e\|^2$ for a dominating function $\mu_A$ of $A$. But the expression $\|A \delta_e\|^2_{G \setminus B_R(e)}$ is, by definition, $\sum_{g \in G \setminus B_R(e)} |(A \delta_e)(g)|^2$ and $(A \delta_e)(g) = a_g$. Now we use that $\|\delta_e\| = 1$ and that for any given $n \in \IN$ we have $\mu_A(R) \le \|A\|_{\mu,n} / R^n$ for all $R > 1$.
\end{proof}

The next is a corollary to Lemma~\ref{lemjk231fsd} and introduces polynomial weights into the derived estimates. Note that the left hand sides of the following estimates are the ones appearing in the definition of the space of rapidly decreasing functions (see, e.g., Jolissaint \cite{jolissaint}), i.e., the following corollary proves that our algebra $\Cpol G$ is continuously contained in this space of rapidly decreasing functions.

\begin{cor}\label{cornkjs23}
Let $A = \sum_{g \in G} a_g g \in \Cpol G$. For every $n \in \IN$ we have
\[\sum_{g \in G} d(g,e)^{2n-2} \cdot |a_g|^2 < \|A\|^2_{\mu,n} \cdot \pi^2 / 6.\]
\end{cor}

\begin{proof}
From the inequality in Lemma~\ref{lemjk231fsd} it follows that
\[\sum_{R \in \IN} \Big( R^{2n-2} \cdot \sum_{g \in G \setminus B_R(e)} |a_g|^2 \Big) < \sum_{R \in \IN} \|A\|^2_{\mu,n} / R^2 \le \|A\|^2_{\mu,n} \cdot \pi^2/6.\]
Now the left hand side of this estimate is equal to $\sum_{g \in G} \sum_{R=1}^{d(g,e)} R^{2n-2} \cdot |a_g|^2$ and we have
\[\sum_{g \in G} d(g,e)^{2n-2} \cdot |a_g|^2 \le \sum_{g \in G} \sum_{R=1}^{d(g,e)} R^{2n-2} \cdot |a_g|^2.\]
This shows the claimed inequality.
\end{proof}

\subsection{Banach space completions}

We let $B^p_r G$ denote the completion of $\IC G \subset \IB(\ell^p G)$ in operator norm for any $p \in [1,\infty]$.

Definition \ref{defnoi34efr} of quasi-local operators in $\IB(\ell^p G)$, Lemma \ref{lemkjdsf0923}, and Definition~\ref{defn92434} carry over and make sense in this setting. For convenience, let us write down the corresponding definition:

\begin{defn}
Let $G$ be a finitely generated group. For every $n \in \IN$ we define a norm on $\IC G \subset \IB(\ell^p G)$ by
\begin{equation*}
\label{eqnj7564k23vd}
\|A\|_{p,\mu,n} := \inf\{D>0\colon \mu^p_A(R) \le D/R^n \ \forall R > 1\},
\end{equation*}
where $\mu^p_A(R)$ denotes the smallest possible dominating function for $A \in \IC G \subset \IB(\ell^p G)$.

We let $\Bpol^p G$ be the closure of $\IC G$ under the family of norms $(\|-\|_{p,\op}, \|-\|_{p,\mu,n})$ for all $n \in \IN$, where $\|-\|_{p,\op}$ denotes the operator norm in $\IB(\ell^p G)$.
\qed
\end{defn}

Estimates~\eqref{eq2334566} and \eqref{eq8934rg5} are still good in this $\ell^p$-setting, and the proof of Proposition~\ref{prop239ds} also goes through without changes. So we have the following fact for every $p \in [1,\infty]$:

\begin{prop}\label{prop89034ew}
$\Bpol^p G$ is a dense and smooth \Frechet subalgebra of $B^p_r G$.
\end{prop}

The kernel estimates from Section~\ref{seckn234323} are also still good in the $\ell^p$-case:

\begin{lem}\label{lem8943e}
Let $A = \sum_{g \in G} a_g g \in \Bpol^p G$ for $p\in[1,\infty)$.

Then for every $n \in \IN$ we have
\[\sum_{g \in G \setminus B_R(e)} |a_g|^p < \|A\|_{p,\mu,n}^p / R^{pn}\]
and also
\[\sum_{g \in G} d(g,e)^{pn-2} \cdot |a_g|^p < \|A\|^p_{p,\mu,n} \cdot \pi^2 / 6.\]
\end{lem}

\begin{rem}\label{remui3489re}
There is a symmetrized version $B_r^{p,\ast} G$, where we complete $\IC G$ simultaneously in the norms $\|-\|_{p,\op}$ and $\|-^\ast\|_{p,\op}$, where $\left(\sum a_g g\right)^\ast = \sum \overline a_g g^{-1}$. Analogously we can define $\Bpol^{p,\ast} G$ and prove that it is a dense and smooth \Frechet $^\ast$-subalgebra of $B_r^{p,\ast} G$.

We have a continuous inclusion $B_r^{p,\ast} G \to B_r^p G$, but except for the abelian case $B_r^{p,\ast} G$ is usually strictly smaller than $B^p_r G$, see Liao--Yu \cite{liao_yu}. In the cited paper it will be also proven that if $G$ has Banach property (RD)$_q$ for $q$ the dual exponent to $p \not= 1$, then the inclusion $B_r^{p,\ast} G \to B_r^p G$ induces isomorphisms on $K$-theory. The usual Property (RD) implies Banach property (RD)$_q$ for all $q \in (1,2)$.

The advantage of $B_r^{p,\ast} G$ over $B_r^p G$ is that due to Banach space interpolation we have a continuous inclusion $B_r^{p,\ast} G \to C_r^\ast G$ and therefore we can try to compare their $K$-theories. Interpolation also gives us continuous inclusions $\Bpol^{p,\ast} G \to \Cpol G$.

Recently Chung \cite{chung} showed the $L_p$-Baum--Connes conjecture with coefficients in $C(X)$ for all $p\in (1,\infty)$, if the group acts with finite dynamic asymptotic dimension on the compact Hausdorff space $X$. This especially implies that the $K$-theory of the corresponding $L_p$-reduced crossed product is independent of $p\in (1,\infty)$.
\qed
\end{rem}

\section{Combinatorics of groups and polynomial contractibility}
\label{secn2339992}

In this section we will firstly discuss the different definitions of higher-order combinatorial functions that we are considering in this paper, and secondly (in Section~\ref{secio23ds}) we will discuss combings of groups, how a polynomially bounded combing produces polynomially contractible groups, and why quasi-geodesically combable groups must have contractible asymptotic cones.

Let $G$ be a countable, discrete group. Recall that $G$ is called being of type $F_N$ if it admits a model for its classifying space $BG$ having a finite $N$-skeleton, and $G$ is called being of type $F_\infty$ if it is of type $F_N$ for every $N \in \IN$.

Note that equivalently we could have said that type $F_N$ means that $G$ admits a model for $BG$ as a simplicial complex with finitely many simplices up to dimension $N$.

Type $F_1$ means finitely generated, and type $F_2$ finitely presented. Being of type $F_\infty$ is a quasi-isometry invariant of groups, see Gromov \cite[Corollary~1.$\text{C}^\prime_2$ on Page~17]{gromov_invariants_infinite}, Alonso \cite{alonso_qi} or Ji--Ramsey \cite[Lemma 2.9]{ji_ramsey}.

\subsection{Higher-order combinatorial functions \`{a} la Ji--Ramsey}

Let us recall the higher-order Dehn functions by Ji--Ramsey \cite{ji_ramsey}.

Let $X$ be a simplicial complex. For a simplicial $N$-boundary $b$ we denote by $l_f(b)$ its filling length, i.e., the least number of $(N+1)$-simplices $a$ with $\partial a = b$. We denote the number of simplices in $b$ by $|b|$. The $N$-th Dehn function $d^N(-)\colon \IN \to \IN \cup \{\infty\}$ of $X$ is now defined as
\[d^N(k) := \sup_{|b| \le k} l_f(b),\]
where the supremum runs over all $N$-boundaries $b$ of $X$ with $|b| \le k$.

For a group $G$ we choose a simplicial model for $BG$. The higher-order Dehn functions of $G$ are then defined as the higher-order Dehn functions of $EG$. If $G$ is of type $F_{N+1}$, then all the higher-order Dehn functions $d^n(-)$ up to $n \le N$ have finite values and the growth type (e.g., being asymptotically a polynomial of a certain degree) does not depend on the chosen model for $BG$ with finite $(N+1)$-skeleton \cite[Section~2]{ji_ramsey}.

\subsection{Higher-order combinatorial functions \`{a} la Riley}

Riley \cite{riley} uses singular combinatorial complexes to define higher-order combinatorial functions. So let us recall the definition:

\begin{defn}[Singular combinatorial complexes]
We will first define combinatorial complexes inductively over the dimension: a $0$-dimensional combinatorial complex is a set with the discrete topology, each point being called both a closed cell and an open cell.

A continuous map $C_1 \to C_2$ is called combinatorial if its restriction to each open cell of $C_1$ is a homeomorphism onto an open cell of $C_2$.

An $N$-dimensional combinatorial complex is a topological space $C$ that can be obtained from a disjoint union $U$ of an $(N-1)$-dimensional combinatorial complex $C^{(N-1)}$ and a collection $(e_\lambda)_{\lambda \in \Lambda}$ of closed $N$-discs in the following way: we suppose the boundaries $\partial e_\lambda$ have combinatorial structures, i.e., for each $e_\lambda$ exists an $(N-1)$-dimensional combinatorial complex $S_\lambda$ with a homeomorphism $\partial e_\lambda \to S_\lambda$. We also suppose that there are combinatorial maps $S_\lambda \to C^{(N-1)}$. Then $C$ is obtained from $U$ by taking the quotient via the attaching maps (equipped with the quotient topology). The open cells of $C$ are defined to be the open cells in $C^{(N-1)}$ and the interiors of the $N$-discs $e_\lambda$. The closed cells of $C$ are the closed cells of the complex $C^{(N-1)}$ together with the $N$-discs $e_\lambda$ equipped with their boundary combinatorial structures $\partial e_\lambda \to S_\lambda$.

To define singular combinatorial complexes we just redefine the maps that are allowed in the inductive definition of combinatorial complexes: a continuous map $C_1 \to C_2$ between singular combinatorial complexes is a singular combinatorial map, if for all $N \in \IN$ each open $N$-cell of $C_1$ is either mapped homeomorphically onto an $N$-cell of $C_2$ or collapses. By the latter we mean that it maps into the image of its boundary.
\qed
\end{defn}

Let $G$ be of type $F_{N+1}$. Then we can construct a compact singular combinatorial $(N+1)$-complex which is the $(N+1)$-skeleton of a model for the classifying space $BG$. We define the $N$th-order two-variable minimal combinatorial isoperimetric function by
\[\delta^{(N)}(n,l) := \sup \{\FVol(\gamma)\colon \gamma \in \Omega_N \text{ with } \Vol(\gamma) \le n \text{ and }\Diam(\gamma) \le l\}\]
and the $N$th-order two-variable minimal combinatorial isodiametric function by
\[\eta^{(N)}(n,l) := \sup \{\FDiam(\gamma)\colon \gamma \in \Omega_N \text{ with } \Vol(\gamma) \le n \text{ and }\Diam(\gamma) \le l\}.\]
Here $\Omega_N$ is the set of singular combinatorial maps $\gamma\colon S^N \to EG^{(N+1)}$, where $S^N$ is given a combinatorial structure (which we do not fix, i.e., for each $\gamma$ we may choose a different combinatorial structure).

The quantity $\Vol(\gamma)$ is the number of $N$-cells in the combinatorial structure of $S^N$ on which $\gamma$ is a homeomorphism, and $\FVol(\gamma)$ is the minimal number amongst the number of $(N+1)$-cells mapped homeomorphically into $EG^{(N+1)}$ of a combinatorial structure of an $(N+1)$-disc $D^{N+1}$ with $\partial D^{N+1} = \gamma$.

Similarly we define the diameter $\Diam(\gamma)$ of $\gamma \in \Omega_N$, resp., its filling diameter $\FDiam(\gamma)$: we endow the $1$-skeleton of $S^N$ (resp., of $D^{N+1}$ for the filling diameter) with a pseudo-metric by defining each edge that is collapsed to a single vertex under $\gamma$ to have length $0$, and length $1$ otherwise.

\begin{example}
Let $N \ge 1$. Riley \cite[Theorem D]{riley} has shown that if $G$ is a finitely generated group whose asymptotic cones\footnote{see the next Section~\ref{secio23ds}} are all $N$-connected, then $G$ is of type $F_{N+1}$ and polynomially contractible up to order $N$.

Recall that the latter means that both the higher-order two-variable isoperimetric and isodiametric functions are polynomially bounded in their variables.

He went on to show \cite[Theorem E]{riley} that virtually nilpotent groups have contractible asymptotic cones, which implies by the above that finitely generated virtually nilpotent groups are of type $F_\infty$ and polynomially contractible.
\qed
\end{example}

Being virtually nilpotent also implies that the higher-order Dehn functions in the sense of Ji--Ramsey are polynomially bounded \cite[Corollary~2.11]{ji_ramsey}. But it is not clear if having contractible asymptotic cones implies that the higher-order Dehn functions of Ji--Ramsey are polynomially bounded (but this would follow if one would show that having polynomially bounded higher-order isoperimetric and isodiametric functions in the sense of Riley is equivalent to having polynomially bounded higher-order Dehn functions in the sense of Ji--Ramsey).

\subsection{Combings of groups and asymptotic cones}
\label{secio23ds}

Let $G$ be a group and choose a generating set for it. We denote by $|\largecdot - \largecdot|$ the distance in the word metric on $G$ derived from this generating set.

A combing of a group $G$ is a mapping
\[\sigma\colon G \to \text{paths in }G, \quad g \mapsto \text{path from }e\text{ to }g\]
(by a path in $G$ we mean a map $\IN_0 \to G$ starting at $e \in G$, becoming eventually constant, and traveling with at most unit speed\footnote{By this we mean that $|\sigma(g)(t) - \sigma(g)(t+1)| \le 1$ for all $t \in \IN_0$.}) with the following property (which is often called ``$k$-fellow-traveling''-property):

There exists $k > 0$ such that for all $g,h \in G$ with $|g-h| \le 1$ and all $t \ge 0$ we have $|\sigma(g)(t) - \sigma(h)(t)| \le k$.

The above type of combing is called a synchronous combing by Gersten \cite{gersten}, and is called a bounded combing by Alonso \cite{alonso}.

\begin{defn}\label{defn023d2d23}
Demanding additional properties on the combing $\sigma$, we arrive at the following notions:
\begin{itemize}
\item In a quasi-geodesically combable group the paths given by the combing $\sigma$ must be quasi-geodesics (for a fixed choice of quasi-geodesicity constants).
\item An automatic group is one which admits a quasi-geodesic combing $\sigma$ that constitutes a regular language.
\item A polynomially combable groups is one where the lengths of the paths\footnote{A length of a path $\sigma(g)$ is the minimal $k \in \IN$, such that $\sigma(g)(t)$ is constant for all $t \ge k$.} $\sigma(g)$ are bounded from above by a polynomial in the length of $g$.
\end{itemize}
This explains half of the notions in Diagram \eqref{eq875562}.
\qed
\end{defn}

\begin{prop}
Let $G$ be a finitely generated group and let $G$ be polynomially combable.

Then $G$ is of type $F_\infty$ and polynomially contractible.
\end{prop}

\begin{proof}
That $G$ is of type $F_\infty$ follows already from the existence of an ordinary combing (i.e., no need for it to be polynomial). This was proven by Alonso \cite[Theorem~2]{alonso}.

That $G$ must have polynomially bounded higher-order Dehn functions was already noticed by Ji--Ramsey \cite[End of 2nd paragraph on p.~257]{ji_ramsey}.

The idea to show that the higher-order combinatorial functions in the sense of Riley are polynomially bounded is the following: fixing any vertex $v_0$ in $EG^{(N+1)}$ and given any map $\gamma\colon S^N \to EG^{(N+1)}$, we use the combing to produce a contraction of this map onto the chosen vertex $v_0$. Due to the $k$-fellow travelling property and since the combing is polynomially bounded we get a bound on the filling volume and filling diameter of $\gamma$ depending firstly, polynomially on the volume and diameter of $\gamma$, and secondly, polynomially on the distance of $\gamma$ to $v_0$. But since a fundamental domain in $EG^{(N+1)}$ is bounded in diameter, we can always find a translate of $v_0$ by deck transformations such that the distance of this translate to $\gamma$ will be uniformly bounded, which finishes the proof of the sought estimate. Note that since we are working here with a fixed vertex $v_0$ and its translates all the time, it suffices that we are given a polynomial combing, i.e., we do not need a polynomial bi-combing.\footnote{In a bi-combing the map $\sigma$ must have the $k$-fellow-travelling property not only for paths starting at the same point $e \in G$, but also for paths starting at neighbouring elements of $g$. To get from $\sigma$ paths that do not start at $e$ we regard a path that $\sigma$ gives as a string of generators producing this path. Then we may apply this string of generators to any other element of the group. There exist groups which are combable but not bi-combable, see Bridson \cite{bridson_combable}.}
\end{proof}

Let us now define asymptotic cones: fixing a metric space $(X,d)$, a choice of
\begin{itemize}
\item a non-principal ultrafilter $\omega$ on $\IN$,
\item a sequence of basepoints $e = (e_n)_{n \in \IN}$ in $X$, and
\item a sequence of strictly positive scaling factors $s = (s_n)_{n \in \IN}$ with $s_n \to \infty$
\end{itemize}
gives us the asymptotic cone
\[\asCone_\omega(X,e,s) := \big( (a_n)_{n \in \IN} \subset X \colon \omegalim_{n \to \infty} \sfrac{1}{s_n} \cdot d(e_n,a_n) < \infty \big) \big/\!\!\sim.\]
The equivalence relation is $(a_n)\!\sim\! (b_n) \Leftrightarrow \omegalim \sfrac{d(a_n,b_n)}{s_n} = 0$ and we define a metric on the asymptotic cone by $\dist((a_n),(b_n)) := \omegalim \sfrac{d(a_n,b_n)}{s_n}$.

Note that if the space $X$ is a finitely generated group with a word metric, then different choices of basepoint sequences lead to isometric asymptotic cones due to the homogeneity of the metric space (see, e.g., Riley~\cite[Lemma~2.2]{riley} for a proof). But the asymptotic cones do depend in general on the choices of non-principal ultrafilter $\omega$ and sequence of scaling factors $(s_n)$.

It is known that the asymptotic cones of hyperbolic groups are $\IR$-trees and therefore contractible (e.g., Gromov \cite[§2.B.(b)]{gromov_invariants_infinite}), and that $\CAT(0)$ groups and co-compact lattices in $\widetilde{\mathrm{SL}(2,\IR)}$ have asymptotic cones which are $\CAT(0)$ spaces and therefore are also contractible (Kar \cite{aditi_kar}). We will generalize the hyperbolic and the $\CAT(0)$ case in the next proposition.

The next seems to be a folklore theorem, cf.~\cite[Paragraph after Corollary~6.6]{BKMM}. We will sketch a rough proof of it.

\begin{prop}
Let $G$ be a quasi-geodesically combable group.

Then the asymptotic cones of $G$ are contractible.
\end{prop}

\begin{proof}
Quasi-geodesic paths in $G$ give Lipschitz paths in the asymptotic cones (to see this use, e.g., the arguments at the beginning of the proof of \cite[Proposition~2.5]{riley}). We need the ``$k$-fellow-traveling''-property to ensure that the constructed Lipschitz paths in the asymptotic cones are well-defined, i.e., do not depend on the chosen sequence of points in $G$ to represent a point in an asymptotic cone. The ``$k$-fellow-traveling''-property furthermore implies that these Lipschitz paths vary continuously (again with a Lipschitz estimate) depending on their starting point.

Hence we can use these Lipschitz paths in the asymptotic cones to contract them to their respective base points.
\end{proof}

\section{Semi-norms on uniformly finite homology}
\label{sec23rds23rw}

Let us first recall the definition of uniformly finite homology:

\begin{defn}[{\cite[Section 2]{block_weinberger_1}}]
Let $X$ be a metric space.

$C_i^\uf(X)$ denotes the vector space of all infinite formal sums $c = \sum a_{\bar x} \bar x$ with $\bar x \in X^{i+1}$ and $a_{\bar x} \in \IC$ satisfying the following three conditions (constants depending on $c$):
\begin{enumerate}
\item There exists $K > 0$ such that $| a_{\bar x} | \le K$ for all $\bar x \in X^{i+1}$.
\item For all $r > 0$ exists $K_r > 0$ with $\card \{ \bar x \in B_r(\bar y) \ | \ a_{\bar x} \not= 0 \} \le K_r$ for all $\bar y \in X^{i+1}$.
\item There is $R > 0$ such that $a_{\bar x} = 0$ if $d(\bar x, \Delta) > R$; $\Delta$ is the multidiagonal in $X^{i+1}$.
\end{enumerate}

The boundary map $\partial \colon C_i^\uf(X) \to C_{i-1}^\uf(X)$ is defined by
\[\partial (x_0, \ldots, x_i) = \sum_{j=0}^i (-1)^j (x_0, \ldots, \hat x_j, \ldots, x_i)\]
and extended by (infinite) linearity to all of $C_i^\uf(X)$. The resulting homology is called the uniformly finite homology $\Huf_\ast(X)$.

If a group $G$ acts by isometries on the space $X$, then we may define the equivariant uniformly finite homology $\HufG_\ast(X)$ by considering only $G$-equivariant chains.
\qed
\end{defn}

Let $G$ be a finitely generated, discrete group. Choosing a finite generating set, we regard $G$ as a metric space under the induced word length. Different choices of generating sets result in quasi-isometric metrics and therefore the following results are independent of this choice. The author proved \cite[Proposition~3.8]{engel_wrongway} the isomorphism $H_\ast(BG) \cong \HufG_\ast(G)$ by exhibiting geometric maps in both directions and which are inverse to each other.\footnote{The chain complex occuring in the definition of $\HufG_\ast(G)$ is the usual bar complex, i.e., this isomorphism is well-known.} Let us describe in the following one of these maps.

If $G$ is of type $F_{N+1}$, we choose a model for $BG$ with a finite $(N+1)$-skeleton. The isomorphism $H_k(BG^{(N+1)}) \cong \HufG_k(G)$ for all $0 \le k \le N$ is then given by the following map: given a simplicial chain in $BG^{(N+1)}$, we lift it equivariantly to $EG^{(N+1)}$. Then we forget everything from the simplices but their vertices, and finally we map these vertices onto an equivariantly and quasi-isometrically embedded copy of $G$ inside $EG^{(N+1)}$. Note that $G$ and $EG^{(N+1)}$ are in this case (i.e., $G$ being of type $F_{N+1}$) even quasi-isometric since $EG^{(N+1)}$ is $G$-finite.

\begin{defn}\label{defnnorm}
Let $G$ be a group and fix $q \in \IN$. For every $n \in \IN$ we define the following norm of an equivariant uniformly finite chain $c = \sum a_{\bar g} \bar g \in \CufG_q(G)$:
\[\|c\|_{n} := \sum_{\bar g \in G^{q+1},\atop\bar g = (e,\ldots)} |a_{\bar g}| \cdot \diam(\bar g)^n,\]
where $\diam(\bar g) := \max_{0 \le k,l \le q} d(g_k,g_l)$.

We equip $\CufG_q(G)$ with the family of norms $(\|-\|_{n} + \|\partial -\|_{n})_{n \in \IN}$, denote its completion to a \Frechet space by $\CpolG_q(G)$ and the resulting homology by $\HpolG_\ast(G)$.
\qed
\end{defn}

In \cite{engel_rough} the author defined groups $\Hufpol_\ast(Y)$ for uniformly discrete metric spaces~$Y$. But note that in this paper we use polynomially weighted $\ell^1$-norms, whereas in the cited paper weighted $\ell^\infty$-norms are used.

\begin{thm}\label{thm092332}
Let $G$ be of type $F_{N+1}$ and polynomially contractible up to order $N$.

Then the map $H_k(BG) \to \HpolG_k(G)$ is an isomorphism for all $0 \le k \le N$.
\end{thm}

\begin{proof}
Since $H_k(BG^{(N+1)}) \cong H_k(BG)$ for all $0 \le k \le N$, we will work in this proof with the complex $K = BG^{(N+1)}$.

Equipping the chain complex $C_\ast(K)$ with the sup-norm, the previously described map $C_\ast(K) \to \CufG_\ast(G)$ inducing the isomorphism $H_\ast(BG) \cong \HufG_\ast(G)$ becomes continuous. Our first step is to construct for $k \le N+1$ continuous chain maps $\Delta_k\colon \CufG_k(G) \to C_k(K)$ which will induce for all $k \le N$ the inverse maps to the maps $H_k(K) \to \HufG_k(G)$.

We fix an equivariant quasi-isometry $X \simeq_{\mathrm{QI}} G$, where $X$ is the universal cover of $K$. Fix $q \le N+1$. Let $\bar g \in G^{q+1}$ with $\bar g = (e,\ldots)$ be given and regard it as a tuple of points in the complex $X$. Let us construct in the following a simplicial map $\Delta_{\bar g}\colon \Delta^q \to X$, where the $q$-simplex $\Delta^q$ consists of at most $P(\diam(\bar g))$ simplices (for some polynomial $P(-)$, which is independent of $\bar g$) and the vertices of the image of $\Delta_{\bar g}$ are exactly $\bar g$: since $K$ is $N$-connected, we can firstly connect the vertices of $\bar g$ to each other, secondly fill the loops that we see by discs, thirdly fill the combinatorial $2$-spheres that we see by balls, etc., up to the point where we have constructed the whole map $\Delta_{\bar g}\colon \Delta^q \to X$; note that we are constructing this map by induction on the skeleta of $\Delta^q$.

Let us now estimate the number of simplices that the domain of this map $\Delta_{\bar g}$ has by using the higher-order combinatorial functions (we will first discuss the more complicated version of these functions as defined by Riley): after the first step, i.e., after having connected the vertices of $\bar g$ to each other, the result is contained in a ball of radius at most $\diam(\bar g)$ and has at most $\binom{q+1}{2} \cdot \diam(\bar g)$-edges. In the second step we are filling the loops by discs, and therefore the number of non-degenerate discs we have at the end is bounded from above by

\[\tbinom{q+1}{3} \cdot \delta^{(1)}\big( 3 \cdot \diam(\bar g), \diam(\bar g) \big)\]
and the whole result is contained in a ball of radius at most
\[\eta^{(1)}\big( 3 \cdot \diam(\bar g), \diam(\bar g) \big) + \diam(\bar g).\]
After the third step, the number of non-degenerate balls we have is bounded by
\[\tbinom{q+1}{4} \cdot \delta^{(2)}\Big( 4 \cdot \delta^{(1)} \big( 3 \diam(\bar g), \diam(\bar g) \big), \eta^{(1)}\big( 3 \cdot \diam(\bar g), \diam(\bar g) \big) + \diam(\bar g) \Big)\]
and the result is contained in a ball of radius at most
\[\eta^{(2)}\Big(4 \cdot \delta^{(1)}\big( 3 \cdot \diam(\bar g), \diam(\bar g) \big), \eta^{(1)}\big( 3 \cdot \diam(\bar g), \diam(\bar g) \big) + \diam(\bar g) \Big) + \diam(\bar g).\]
The concrete number of non-degenerate simplices that we get at the end is not important to us. The important part is the following: at the $k$th step we get estimates which arise by plugging the estimates from the $(k-1)$st step (slighly modified) into the functions $\delta^{(k)}(-,-)$ and $\eta^{(k)}(-,-)$. Since we assume that all these functions are polynomially bounded and since we start in the first step with just plugging in $\diam(\bar g)$ into these functions, we conclude that the number of simplices that the domain of the map $\Delta_{\bar g}$ has is bounded from above by a polynomial in $\diam(\bar g)$. The same argumentation also works for the higher-order Dehn functions as defined by Ji--Ramsey (the argument is even easier since we do not have to keep track of the diameters).

The above procedure gives us the claimed map $\Delta_k\colon \CufG_k(G) \to C_k(K)$: for a given chain from $\CufG_k(G)$ we construct these simplicial maps $\Delta_{\bar g}$ into $X$ (where $\bar g$ runs over the simplices from the given chain), and we can now push them down to $K$ to get an element of $C_k(K)$. We can do this up to $k = N+1$, and they will be chain maps since we used the inductive procedure to construct them. Because of our above arguments on the number of simplices in the maps $\Delta_{\bar g}$ we see that these maps $\Delta_k$ will be continuous if we use on $\CufG_k(G)$ a norm $\|-\|_n$ with $n \in \IN$ big enough. Therefore we can extend $\Delta_k$ continuously to a map $\CpolG_k(G) \to C_k(K)$ for all $k \le N+1$.

The composition $C_k(K) \to \CpolG_k(G) \xrightarrow{\Delta_k} C_k(K)$ is almost the identity map: the error occurs only because we use the equivariant quasi-isometry $X \simeq_{\mathrm{QI}} G$ in between. But it is chain homotopic to the identity on $C_k(K)$.

The other composition $\CpolG_k(G) \xrightarrow{\Delta_k} C_k(K) \to \CpolG_k(G)$ does the following: chains in $\CpolG_k(G)$ can contain big simplices (i.e., with a big diameter), but after applying this composition we get a chain which consists only of simplices of edge length $1$. If we consider this map on the domain (and with target) $\CufG_k(G)$, then the composition is chain homotopic to the identity. A similar counting argument as above gives that the chain homotopy is continuous and therefore extends to $\CpolG_k(G)$ to show that the composition in question is chain homotopic to the identity on $\CpolG_k(G)$.

Since chain homotopies map up by one degree, this is the reason why at the end we only get the isomorphism $H_k(K) \to \HufG_k(G)$ up to $k \le N$.
\end{proof}

\begin{cor}\label{corj9809uk23ds}
Let $G$ be of type $F_\infty$ and polynomially contractible.

Then the map $\HufG_\ast(G) \to \HpolG_\ast(G)$ is an isomorphism.
\end{cor}

\begin{proof}
From Theorem~\ref{thm092332} it follows that $H_\ast(BG) \to \HpolG_\ast(G)$ is an isomorphism. Now we just combine this with the isomorphism $H_\ast(BG) \cong \HufG_\ast(G)$.
\end{proof}

\section{Character maps for group rings and the main diagram}
\label{seco23324545}

Let us explain the maps and notation occuring in the main diagram \eqref{eqnjk23d} of the introduction. For convenience, we have reproduced the diagram here:
\[\xymatrix{
RK_\ast(BG) \ar[rr]^{\mu^\alg} \ar[dd]_{\ch_n} & & K^\alg_\ast(\calS G) \ar[r] \ar[d]_{\trace \circ \ch_n^{\calS G}} & K_\ast(B_r^p G) \ar[d]\\
& & \HC_n(\IC G) \ar[r] \ar[d]_{\chi_n} & \HCocont_n(\Bpol^p G) \ar@{-->}[d]\\
H_{n}(BG) \ar[rr] & & \HufG_{n}(G) \ar[r] & \HpolG_{n}(G)}\]

The map $\mu^\alg$ is the so-called algebraic Baum--Connes assembly map. If we denote by $\calS = \bigcup_{p \ge 1} \calS^p(H)$ the Schatten class operators on some fixed, separable, $\infty$-dimensional Hilbert space $H$, then $\mu^\alg$ is defined as the Farrell--Jones assembly map for the ring $\calS$. It was first investigated by Yu \cite{yu_algebraic_novikov}, who showed that it is always rationally injective. See the discussion in \cite[Section 2]{EM_burghelea} for more information about this. The algebraic Baum--Connes assembly map factors the usual Baum--Connes assembly map.

The map $\ch_n\colon RK_\ast(BG) \to H_{n}(BG)$ is the usual homological Chern character in degree $n \in \IN_0$. For the homology groups we use complex coefficients (since the map $\chi_n$ will map into complex coefficients).

The map $H_{n}(BG) \to \HufG_{n}(G)$ is the one explained at the beginning of Section~\ref{sec23rds23rw}. By the author's result \cite[Proposition~3.8]{engel_wrongway} it is always an isomorphism. Since the Chern character is known to be rationally an isomorphism (if we map into the direct sum of all degrees $n \in \IN_0$ at once), the Banach strong Novikov conjecture (for classes of degree $n$) therefore follows (together with constructing the dotted arrow in the diagram) from showing that the map $\HufG_{n}(G) \to \HpolG_{n}(G)$, which is induced from completing the corresponding chain complex, is injective. In Corollary~\ref{corj9809uk23ds} we even managed to show that $\HufG_{n}(G) \to \HpolG_{n}(G)$ is bijective for polynomially contractible groups.

Let us discuss now the map $K^\alg_\ast(\calS G) \to \HC_n(\IC G)$. From results of Corti\~{n}as--Thom \cite[Theorems 6.5.3 \& 8.2.5]{CT} we conclude that we have $K^\alg_\ast(\calS G) \cong \KH_\ast(\calS G)$ and these groups are $2$-periodic (and the periodicity is induced by multiplication with the Bott element). Now since $\calS G$ is the directed limit of $\cdots \hookrightarrow \calS^p G \hookrightarrow \calS^{p+1} G \hookrightarrow \cdots$, since $\KH$-theory commutes with directed limits, and since from the proof of Corti\~{n}as--Tartaglia \cite[Corollary 3.5]{ctartaglia} we infer that we have isomorphisms $\KH_\ast(\calS^p G) \to \KH_\ast(\calS^{p+1} G)$ for $p \ge 1$, we conclude $\KH_\ast(\calS G) \cong \KH_\ast(\calS^1 G)$. We can use now the Connes--Karoubi character $\KH_\ast(\calS^1 G) \to \HP_\ast^\mathrm{cont}(\calS^1 G)$, and the whole composition of all the above together with the map $\HP_\ast^\mathrm{cont}(\calS^1 G) \to \HC_n^\mathrm{cont}(\calS^1 G)$ is denoted by $\ch_n^{\calS G}$ in the diagram. Then we can use the trace to get to $\HC_n(\IC G)$.

Since $\Bpol^p G \subset B_r^p G$ is a smooth and dense sub-algebra, we have $K_\ast(\Bpol^p G) \cong K_\ast(B_r^p G)$. The Connes--Karoubi character gives us the morphism $K_\ast(\Bpol^p G) \to \HCocont_n(\Bpol^p G)$. The map $\HC_n(\IC G) \to \HCocont_n(\Bpol^p G)$ is induced from the inclusion $\IC G \to \Bpol^p G$, and the morphism $K^\alg_\ast(\calS G) \to K_\ast(B_r^p G)$ is induced from passing to the completion of algebras $\calS G \cong \calS \otimes_\alg \IC G \to \IK(H) \otimes B_r^p G$ and since $\IK(H) \otimes -$ can be ignored in top.\ $K$-theory.

\subsection{The character map \texorpdfstring{$\chi$}{X} and its continuity}
\label{secji34red}

Let us define the map $\chi_n\colon \HC_n(\IC G) \to \HufG_{n}(G)$. Let therefore $A_0 \otimes \cdots \otimes A_n \in \IC G^{\, \otimes(n+1)}$ be given and we set $\chi_n (A_0 \otimes \cdots \otimes A_n) \in \CufG_n(G)$ as
\[\chi_n(A_0 \otimes \cdots \otimes A_n)(g_0, \ldots, g_n) := \frac{1}{(n+1)!} \sum_{\sigma \in \mathfrak{S}_{n+1}} (-1)^\sigma A_0\big(g_{\sigma(n)}^{-1} g_{\sigma(0)}\big) \cdots A_n\big(g_{\sigma(n-1)}^{-1} g_{\sigma(n)}\big).\]
It is straight-forward to verify that the above formula descends to a chain map $\chi_\bullet$ on the cyclic complex $C_\bullet^\lambda(\IC G)$ and so induces a map $\chi_n\colon \HC_n(\IC G) \to \HufG_n(G)$.

In the following proposition we will investigate the continuity of the maps $\chi_\ast$. Then we may define the map $\HCocont_n(\Bpol^p G) \to \HpolG_{n}(G)$ by continuous extension. Note that in order for $\chi_n$ to descend to $\HCocont_n(\Bpol^p G)$ we must have that $\chi_{n+1}$ is also continuous (to guarantee compatibility with the boundary operator in the chain complexes).

\begin{prop}\label{prop3490we}
The map $\chi_n\colon \IC G^{\, \otimes(n+1)} \to \CufG_n(G)$ is continuous against the topology induced from $\Bpol^p G$ for every $p \le \sfrac{n+1}{n}$.

Let $G$ have Property (RD). Then for every $n \in \IN_0$ the map $\chi_n$ is continuous against the topology coming from $\Cpol G$.
\end{prop}

\begin{proof}
Recall Definition~\ref{defnnorm} of the norms we are using on uniformly finite homology. We fix $k \in \IN$ and first note that $\diam(\bar g)^k \le C_{k,n} \cdot \big( d(e,g_1)^k + d(g_1, g_2)^k + \cdots + d(g_n,e)^k \big)$, where the constant $C_{k,n}$ only depends on $k,n \in \IN$.

So we have to estimate $\sum_{\bar g \in G^{n+1}\atop\bar g = (e,\ldots)} \big|A_0\big(g_{\sigma(n)}^{-1} g_{\sigma(0)}\big) \cdots A_n\big(g_{\sigma(n-1)}^{-1} g_{\sigma(n)}\big)\big| \cdot d(g_i,g_{i+1})^k$ for all $0 \le i \le n-1$ to finish this proof (we write $g_0 = e$ and we have fixed a permutation $\sigma$). We have $d(g_i,g_{i+1}) = d(g_{\sigma(s)}, g_{\sigma(t)})$ for certain $s$ and $t$, and we use now the triangle inequality to write $d(g_{\sigma(s)}, g_{\sigma(t)}) \le d(g_{\sigma(s)}, g_{\sigma(s+1)}) + d(g_{\sigma(s+1)}, g_{\sigma(s+2)}) + \cdots + d(g_{\sigma(t-1)}, g_{\sigma(t)}).$ So we must find an estimate for $\sum_{\bar g \in G^{n+1}\atop\bar g = (e,\ldots)} \big|A_0\big(g_{\sigma(n)}^{-1} g_{\sigma(0)}\big) \cdots A_n\big(g_{\sigma(n-1)}^{-1} g_{\sigma(n)}\big)\big| \cdot d(g_{\sigma(s)}, g_{\sigma(s+1)})^k$.

Noting that $d(g_{\sigma(s)}, g_{\sigma(s+1)}) = d\big(g^{-1}_{\sigma(s+1)} g_{\sigma(s)},e\big)$ we can rewrite the expression we have to estimate as the iterated convolution $\big( |A_{\sigma^{-1}(0)}| \ast \cdots \ast |A_s| \cdot d(-,e)^k \ast \cdots \ast |A_{\sigma^{-1}(n)}| \big) (e)$, where $|A_i| = \sum |a_g| g$ if $A_i = \sum a_g g$.

We can estimate the value of this iterated convolution at $e \in G$ by its sup-norm, and this in turn can be estimated by an iterated application of Young's inequality\footnote{Young's inequality does hold true on locally compact (unimodular) groups. A proof may be found in, e.g., Hewitt--Ross \cite[Theorem 20.18 on Page 296]{hewitt_ross}. The cited reference does not include the extremal case of the sup-norm on the left-hand side of the inequality, but it is also true.}:
\[\big\| |A_{\sigma^{-1}(0)}| \ast \cdots \ast |A_s| \cdot d(-,e)^k \ast \cdots \ast |A_{\sigma^{-1}(n)}| \big\|_\infty \le \prod_{i = 0, \ldots, n\atop i \not= s} \big\| |A_i| \big\|_{p^\prime} \cdot \big\| |A_s| \cdot d(-,e)^k \big\|_{p^\prime},\]
where $p^\prime = \sfrac{n+1}{n}$. Therefore, if the operators $A_0, \ldots, A_n$ are all from $\Bpol^p G$ for a $p \le p^\prime$ we can finish the proof of the first claim of the proposition by applying Lemma~\ref{lem8943e}.

Let $G$ now have Property (RD). By definition this means that we have a continuous inclusion of the space $H^\infty G$ of rapidly decreasing functions into $C_r^\ast G$. By Corollary~\ref{cornkjs23} we have a continuous inclusion of $\Cpol G$ into $H^\infty G$. Now we also note that on $H^\infty G$ the operation of taking the absolute value of the coefficients of a function is an isometry (but note that it is not a linear operation). So we have a chain of continuous maps
\begin{equation}\label{eqimnb23e}
\Cpol G \to H^\infty G \xrightarrow{|-|} H^\infty G \to C_r^\ast G.
\end{equation}

We want to estimate $\big( |A_{\sigma^{-1}(0)}| \ast \cdots \ast |A_s| \cdot d(-,e)^k \ast \cdots \ast |A_{\sigma^{-1}(n)}| \big) (e)$. This is at most the $\ell^2$-norm of this iterated convolution, and this can be now estimated by
\[\prod_{i = 0, \ldots, n\atop i \not= s,\sigma^{-1}(n)} \big\| |A_i| \|_\op \cdot \big\| |A_s| \cdot d(-,e)^k \big\|_\op \cdot \big\| |A_{\sigma^{-1}(n)}| \big\|_2.\]
By \eqref{eqimnb23e} the operator norm of $|A_i|$ is bounded by certain $\Cpol G$-norms of $A_i$, and the same holds for the operator norm of $|A_s| \cdot d(-,e)^k$, and for the $\ell^2$-norm of $|A_{\sigma^{-1}(n)}|$. This finishes the second claim of the proposition.
\end{proof}

\bibliography{./Banach_SNC_pol_contr}
\bibliographystyle{amsalpha}

\end{document}